\documentclass[12pt,a4]{amsart}
\usepackage[a4paper, left=28mm, right=28mm, top=28mm, bottom=34mm]{geometry}
\usepackage[all]{xy}
\usepackage{amsmath}
\usepackage{amssymb}
\usepackage{amsthm}
\usepackage{mathrsfs}
\usepackage{url}
\usepackage{comment}

\theoremstyle{definition}
\newtheorem{thm}{Theorem}[section]
\newtheorem{lem}[thm]{Lemma}
\newtheorem{cor}[thm]{Corollary}
\newtheorem{prop}[thm]{Proposition}

\theoremstyle{definition}

\newtheorem{rem}[thm]{Remark}

\newtheorem{condition}[thm]{Condition}

\newtheorem{ex}[thm]{Example}

\numberwithin{equation}{section}

\def\F{{\mathbb F}}

\def\Q{{\mathbb Q}}
\def\R{{\mathbb R}}
\def\Z{{\mathbb Z}}

\def\PP{{\mathbb P}}

\def\Aut{\mathop{\mathrm{Aut}}\nolimits}

\def\Frob{\mathop{\mathrm{Frob}}\nolimits}
\def\Frac{\mathop{\mathrm{Frac}}\nolimits}

\def\Gal{\mathop{\mathrm{Gal}}\nolimits}

\def\Jac{\mathop{\mathrm{Jac}}\nolimits}

\def\Ker{\mathop{\mathrm{Ker}}\nolimits}
\def\Res{\mathop{\mathrm{Res}}\nolimits}

\def\AGL{\mathop{\mathrm{AGL}}\nolimits}
\def\GL{\mathop{\mathrm{GL}}\nolimits}
\def\SL{\mathop{\mathrm{SL}}\nolimits}

\def\deg{\mathop{\text{\rm deg}}\nolimits}

\def\det{\mathop{\mathrm{det}}\nolimits}
\def\chara{\mathop{\mathrm{char}}}
\def\divi{\mathop{\mathrm{div}}}

\newcommand{\transp}[1]{{}^{t}\!{#1}}

\begin{document}

\title[The Hasse principle for finite Galois modules]{The Hasse principle for finite Galois modules allowing exceptional sets of positive density}

\author[Y.\ Ishitsuka]{Yasuhiro Ishitsuka}
\address{
Institute of Mathematics for Industry, Kyushu University,
Fukuoka, 819-0395, Japan}
\email{yishi1093@gmail.com}

\author[T.\ Ito]{Tetsushi Ito}
\address{Department of Mathematics, Faculty of Science, Kyoto University, Kyoto 606--8502, Japan}
\email{tetsushi@math.kyoto-u.ac.jp}

\date{February 17, 2022}


\begin{abstract}
We study a variant of the Hasse principle for finite Galois modules,
allowing exceptional sets of positive density.
For a Galois module whose underlying abelian group
is isomorphic to $\F_p^{\oplus r}$ ($r \leq 2$),
we show that the product of the restriction maps for places
in a set of places $S$ is injective if the Dirichlet density of
$S$ is strictly larger than $1 - p^{-r}$.
We give applications to
the local-global divisibility problem for elliptic curves
and the Hasse principle for flexes on plane cubic curves.
\end{abstract}

\maketitle

\section{Introduction}

Let $K$ be a global field (i.e.\ a finite extension of $\Q$ or $\F_p(T)$
for a prime number $p$).
Let $M$ be a finite abelian group equipped with
a continuous action of $\Gal(K^{\mathrm{sep}}/K)$,
where $K^{\mathrm{sep}}$ is a separable closure of $K$.
For simplicity, we call it a \emph{finite Galois module} over $K$.
The \emph{Hasse principle} asks whether the restriction maps
\[
\Res_{M} \colon
H^1(\Gal(K^{\mathrm{sep}}/K), M) \to
\prod_{v \, : \, \text{place of $K$}} H^1(\Gal(K_v^{\mathrm{sep}}/K_v), M)
\]
is injective.
Here $K_v$ is the completion of $K$ at $v$.
Although $\Res_{M}$ is not injective in general,
it is known to be injective for a number of cases.
For example, if the underlying abelian group of $M$ is
is isomorphic to $(\Z/p^s \Z)^{\oplus r}$,
the map $\Res_{M}$ is known to be injective
if one of the following conditions is satisfied.
\begin{condition}
\label{Condition:prs}
\begin{itemize}
\item $p$ is odd, $r = 1$, and $s \geq 1$ is arbitrary.
\item $p$ is arbitrary and $(r, s) = (2, 1)$.
\item $(p, r, s) = (2, 1, 1)$, $(2, 1, 2)$, or $(2, 3, 1)$.
\end{itemize}
\end{condition}
In fact, when $p \neq \chara K$ and  $(r,s) = (2,1)$,
the injectivity of $\Res_{M}$ follows from
\cite[Chapter I, Theorem 4.10 (a)]{Milne}
and \cite[Chapter I, Corollary 9.6]{Milne}.
Other cases are similar and well-known.

In this paper, we shall prove a variant of the above results
allowing exceptional sets of positive density,
following a suggestion by J.-P.\ Serre.
Let $S$ be a set of finite places of $K$.
We may ask whether the product of the restriction maps for places in $S$
\[
\Res_{M,S} \colon
H^1(\Gal(K^{\mathrm{sep}}/K), M) \to
\prod_{v \in S} H^1(\Gal(K_v^{\mathrm{sep}}/K_v), M)
\]
is injective or not.

The main theorem of this paper is as follows.

\begin{thm}
\label{Theorem:GaloisCohomology}
Let $K$ be a global field.
Let $M$ be a finite Galois module over $K$
whose underlying abelian group is isomorphic to $(\Z/p^s \Z)^{\oplus r}$.
Assume that $(p,r,s)$ satisfies one of Condition \ref{Condition:prs},
and the Dirichlet density $\delta(S)$ exists and satisfies
\[ \delta(S) > 1 - p^{-rs}. \]
Then $\Res_{M,S}$ is injective.
\end{thm}

Recall that the \emph{Dirichlet density} of $S$ is defined by
\[
\delta(S) := \lim_{s \to 1^{+}}
\frac{\sum_{v \in S} |\kappa(v)|^{-s}}{\sum_{v \, :\,  \text{finite place of $K$}}  |\kappa(v)|^{-s}}
\]
if it exists.
When $K$ is a number field, we have the same result for the natural density
instead of the Dirichlet density;
see Remark \ref{Remark:density}.

We shall also prove that the above result is optimal.

\begin{thm}
\label{Theorem:Counterexample}
Let $K$ be a global field.
Let $p$ be a prime number, and $r,s \geq 1$ positive integers.
\begin{enumerate}
\item
Assume that $(p,r,s)$ does \emph{not} satisfy one of Condition \ref{Condition:prs}.
Then there is a finite Galois module $M$ over $K$ such that
$M \cong (\Z/p^s \Z)^{\oplus r}$ as an abelian group,
and
\[
\Res_{M} \colon
H^1(\Gal(K^{\mathrm{sep}}/K), M) \to
\prod_{v \,:\, \text{place of}\ K} H^1(\Gal(K_v^{\mathrm{sep}}/K_v), M)
\]
is \emph{not} injective.

\item
Assume that $(p,r,s)$ satisfies one of Condition \ref{Condition:prs}.
Then there is a finite Galois module $M$ over $K$
and a set $S$ of places of $K$ such that
$M \cong (\Z/p^s \Z)^{\oplus r}$ as an abelian group,
$\delta(S) = 1 - p^{-rs}$,
and $\Res_{M,S}$ is \emph{not} injective.
\end{enumerate}
\end{thm}

Theorem \ref{Theorem:Counterexample} (1) means
the assumptions on $(p,r,s)$ cannot be weakened in general.
Theorem \ref{Theorem:Counterexample} (2) means
the lower bound ``$\delta(S) > 1 - p^{-rs}$'' cannot be improved
in general.

We shall give two applications.
As a first application, we shall show the following corollary
on the local-global divisibility problem.
(Here we only state the results when $p$ is odd.
The case $p = 2$ is similar and omitted.)

\begin{cor}
\label{Corollary:LocalGlobalDivisibility}
Let $K$ be a global field.
Let $p$ be an odd prime number with $p \neq \chara K$.
Let $G$ be an elliptic curve or an algebraic torus of dimension $\leq 2$ over $K$.
Let $r,s \geq 1$ be positive integers
satisfying the following conditions.
\begin{itemize}
\item If $G$ is an elliptic curve or an algebraic torus of dimension $2$, then $(r,s) = (2,1)$.
\item If $G$ is an algebraic torus of dimension $1$, then $r = 1$ and $s$ is arbitrary.
\end{itemize}
Let $Q \in G(K)$ be a $K$-rational point.
Let $S$ be a set of finite places of $K$ satisfying $\delta(S) > 1 - p^{-rs}$.
If $Q \in p^s G(K_v)$ for every $v \in S$,
then $Q \in p^s G(K)$.
\end{cor}

Recall that Kummer theory gives embeddings
\begin{align*}
G(K)/p^s G(K) &\hookrightarrow H^1(\Gal(K^{\mathrm{sep}}/K), G[p^s](K^{\mathrm{sep}})), \\
G(K_v)/p^s G(K_v) &\hookrightarrow H^1(\Gal(K_v^{\mathrm{sep}}/K_v), G[p^s](K_v^{\mathrm{sep}})).
\end{align*}
Therefore, Corollary \ref{Corollary:LocalGlobalDivisibility}
immediately follows from Theorem \ref{Theorem:GaloisCohomology}.

\begin{rem}
Usually, the local-global divisibility problem
is studied for a set $S$ of places
with Dirichlet density $1$ (i.e.\ $\delta(S) = 1$).
In this case, Corollary \ref{Corollary:LocalGlobalDivisibility}
is well-known.
When $G = \mathbb{G}_m$ is the multiplicative group,
it is known as the Grunwald--Wang theorem
\cite[Chapter IX, Theorem 9.1.11]{NeukirchSchmidtWingberg}.
When $G$ is an elliptic curve,
it is proved by Dvornicich--Zannier \cite[Theorem 3.1]{DvornicichZannier}
and Wong \cite[Theorem 1]{Wong}.
\end{rem}

As a second application,
we study the Hasse principle for flexes on cubic curves.
Note that the Hasse principle for flexes follows easily from
the classical results in the literature (e.g., \cite{Milne}).
In fact, for a smooth cubic curve $C \subset \PP^2$ over a global field $K$ with $\chara K \neq 3$, the flexes on $C$ give
an element of the Galois cohomology
$H^1(\Gal(K^{\mathrm{sep}}/K), M)$
for $M = \Jac(C)[3](K^{\mathrm{sep}})$.
Since
$M \cong \F_3^{\oplus 2}$ as an abelian group,
the map $\Res_{M}$ is injective
by \cite[Chapter I, Theorem 4.10 (a)]{Milne}, \cite[Chapter I, Corollary 9.6]{Milne}.
The Hasse principle for flexes follows.
We shall improve this result using Theorem \ref{Theorem:GaloisCohomology},
and prove that a smooth cubic curve over a global field $K$ with $\chara K \neq 3$ has a flex over $K$ if and only if
it has a flex locally at a set of places of Dirichlet density
strictly larger than $8/9$.
We shall also construct explicit examples of cubic curves attaining the bound ``$8/9$.''
We also prove similar results over global fields of characteristic $3$.
See Section \ref{Section:Flex} for details.

\begin{rem}
\label{Remark:Serre}
In the first
version of this paper\footnote{\url{https://arxiv.org/abs/1807.09909v1}},
we gave a proof of the Hasse principle for flexes
based on a computer calculation on subgroups of $\AGL_2(\F_3)$.
The proof in the current version
is due to J.-P.\ Serre,
which does not require computer calculation.
He kindly suggested to study
the Hasse principle allowing exceptional sets of positive density.
Theorem \ref{Theorem:GaloisCohomology}
and Theorem \ref{Theorem:Counterexample}
of this paper answer
the question raised by him.
\end{rem}

Let us explain the outline of the proof.
Our proof of Theorem \ref{Theorem:GaloisCohomology}
and
Theorem \ref{Theorem:Counterexample}
is of group theoretic nature.
Let
\[ \rho \colon \Gal(K^{\mathrm{sep}}/K) \to \Aut(M) \cong \GL_r(\Z/p^s \Z) \]
be the continuous homomorphism corresponding to $M$.
An element
\[ \alpha \in H^1(\Gal(K^{\mathrm{sep}}/K), M) \]
corresponds to a $1$-cocycle
\[ \widetilde{\alpha} \colon \Gal(K^{\mathrm{sep}}/K) \to M \cong (\Z/p^s \Z)^{\oplus r}. \]
It satisfies the cocycle condition
$\widetilde{\alpha}(g_1 g_2) = \widetilde{\alpha}(g_1) + g_1 \cdot \widetilde{\alpha}(g_2)$.
For each $g \in \Gal(K^{\mathrm{sep}}/K)$,
we have an affine transformation on $(\Z/p^s \Z)^{\oplus r}$ given by
\[ (\Z/p^s \Z)^{\oplus r} \to (\Z/p^s \Z)^{\oplus r}, \quad x \mapsto \rho(g) \cdot x + \widetilde{\alpha}(g). \]
A direct calculation shows the map
\[ \psi \colon \Gal(K^{\mathrm{sep}}/K) \to \AGL_r(\Z/p^s \Z) := (\Z/p^s \Z)^{\oplus r} \rtimes \GL_r(\Z/p^s \Z) \]
defined by $\psi(g) := (\widetilde{\alpha}(g),\,\rho(g))$
is a group homomorphism.
The element $\alpha$ is trivial if and only if
the action of $\psi(g)$ on $(\Z/p^s \Z)^{\oplus r}$
(as an affine transformation)
has a fixed point; see Lemma \ref{Lemma:AffineFixedPoints}.
Therefore, we can interpret
the assertion of Theorem \ref{Theorem:GaloisCohomology}
in terms of group theoretic properties of
the image $\psi(\Gal(K^{\mathrm{sep}}/K))$.
We shall prove it using the Chebotarev density theorem.

\begin{comment}
This work was motivated by the classical Hasse principle for rational points on
algebraic varieties.
In 1951, Selmer proved the smooth cubic curve
$3 X^3 + 4 Y^3 + 5 Z^3 = 0$
has no $\Q$-rational point, but it has $\R$-rational points and $\Q_p$-rational points for all $p$ \cite{Selmer}.
Since then, many people are working on the Hasse principle
in various geometric settings;
see \cite{Bhargava:HassePrinciple}, \cite{JahnelLoughran}, \cite{IIOTU} for example.

A natural question is to ask whether the results of this paper
can be generalized to plane curves of degree $\geq 4$.
For smooth quartics (i.e.,\ smooth plane curves of degree $4$),
natural generalizations of flexes are bitangents.
In \cite{IIOTU}, it is proved that the bitangents of a smooth quartic
over a global field of characteristic different from $2$
do not satisfy the Hasse principle.
An algorithm to construct quartics failing the Hasse principle for bitangents via del Pezzo surfaces of degree $2$ is also given;
see \cite{IIOTU} for details.
\end{comment}

The outline of this paper is as follows.
In Section \ref{Section:GroupTheoreticProperties},
we prove group theoretic results on affine transformation groups.
In Section \ref{Section:ProofMainTheorem},
we prove Theorem \ref{Theorem:GaloisCohomology}.
In Section \ref{Section:Counterexample},
we prove Theorem \ref{Theorem:Counterexample}.
In Section \ref{Section:Flex},
we give an application to the Hasse principle for flexes on cubic curves.
Finally, in Section \ref{Section:FlexHessian},
we summarize results on flexes and Hessian curves.

\begin{rem}
Condition \ref{Condition:prs} is related to the cyclicity
of a Sylow $p$-subgroup of $\GL_r(\F_p)$;
see the proof of Proposition \ref{Proposition:AffineTransformation}.
For any $p$, any Sylow $p$-subgroup of $\GL_2(\F_p)$ is cyclic of order $p$.
When $p$ is odd,
the Sylow $p$-subgroup of $(\Z/p^s \Z)^{\times}$ is cyclic for any $s \geq 1$; it is trivial when $s = 1$.
When $p = 2$,
the Sylow $2$-subgroup of $(\Z/2^s \Z)^{\times}$ is cyclic
if and only if $s \leq 2$; it is trivial when $s = 1$.
The case $(p,r,s) = (2,3,1)$ is exceptional;
although a Sylow $2$-subgroup of $\GL_3(\F_2)$ is not cyclic,
the Hasse principle holds.
(Dvornicich--Zannier observed this fact; see \cite[Remark 3.6]{DvornicichZannier}.)
\end{rem}

\section{Group theoretic results}
\label{Section:GroupTheoreticProperties}

\subsection{Affine transformation groups}

In this section, we give certain group theoretic results on
affine transformation groups.

Let $R$ be a commutative ring, and $R^{\oplus r}$
a free $R$-module of rank $r \geq 1$.
Let $e_1,\ldots,e_r \in R^{\oplus r}$ be the standard basis.
Let $\AGL_r(R)$ be the affine transformation group of $R^{\oplus r}$.
It is a semi-direct product of the general linear group $\GL_r(R)$
and the group of translations $R^{\oplus r}$, i.e.,
we have
\[ \AGL_r(R) := R^{\oplus r} \rtimes \GL_r(R). \]
The projection to the second factor is denoted by
\[ \pi \colon \AGL_r(R) \to \GL_r(R). \]
Let $g_{v,m} \in \AGL_r(R)$ be the element corresponding to
a pair $(v,m) \in R^{\oplus r} \rtimes \GL_r(R)$.
For $g_{v,m}, g_{v',m'} \in \AGL_r(R)$,
the multiplication is given by
$g_{v,m} \cdot g_{v',m'} = g_{v + m(v'),\, m m'}$.
The action of $g_{v,m}$ on $w \in R^{\oplus r}$ is given by
$g_{v,m} \cdot w = v + m(w)$.
We say an element $g \in \AGL_r(R)$ is \emph{fixed point free}
if the action of $g$ on $R^{\oplus r}$
has no fixed point.
An element $g_{v,m} \in \AGL_r(R)$ is fixed point free if and only if $v \notin \mathrm{Im}(I_r - m)$,
where $I_r \in \GL_r(R)$ is the identity matrix.

For a finite ring $R$ (e.g., $R = \Z/p^s \Z$)
and a subgroup $G \subset \AGL_r(R)$,
the proportion of fixed point free elements in $G$ is denoted by
\[ d_G := | \{ g \in G \mid g \ \text{is fixed point free} \}|/|G|. \]
If the natural action of $G$ on $(\Z/p^s \Z)^{\oplus r}$
has a fixed point,
then $d_G = 0$.
But the converse does not hold in general.
Nevertheless, we have the following results.

\begin{prop}
\label{Proposition:AffineTransformation}
Assume that $(p,r,s)$ satisfies one of Condition \ref{Condition:prs}.
Let $G \subset \AGL_r(\Z/p^s \Z)$ be a subgroup with $d_G = 0$.
Then the natural action of $G$ on $(\Z/p^s \Z)^{\oplus r}$
has a fixed point.
\end{prop}

If $(p,r,s) = (2,3,1)$,
the assertion can be verified by \texttt{GAP}.
(For a proof which does not require computer verification,
see Remark \ref{Remark:p=2}.)

In the following, we shall only consider the case $r \leq 2$.
In this case, Condition \ref{Condition:prs} implies
any $p$-Sylow subgroup of $\GL_r(\Z/p^s \Z)$ is cyclic.

There are two actions of $\AGL_r(\Z/p^s\Z)$
on $(\Z/p^s\Z)^{\oplus r}$.
One is the natural action as affine transformations.
The other one is the action via the projection $\pi$.
We call the latter the \emph{linear action}.
In order to distinguish them,
the free $\Z/p^s\Z$-module
$(\Z/p^s\Z)^{\oplus r}$ equipped with the linear action of $\AGL_r(\Z/p^s\Z)$
is denoted by $(\Z/p^s\Z)^{\oplus r}_{\mathrm{lin}}$.

\begin{lem}
\label{Lemma:AffineFixedPoints}
Let $G \subset \AGL_r(\Z/p^s\Z)$ be a subgroup. 
\begin{enumerate}
\item 
The composition of the embedding $G \hookrightarrow \AGL_r(\Z/p^s\Z)$ and the projection to the first factor
\[ \AGL_r(\Z/p^s\Z) = (\Z/p^s\Z)^{\oplus r} \rtimes \GL_r(\Z/p^s\Z) \to (\Z/p^s\Z)^{\oplus r}_{\mathrm{lin}} \]
is a 1-cocycle if we consider the linear action of $G$.

\item
Let $\alpha_G \in H^1(G, (\Z/p^s\Z)^{\oplus r}_{\mathrm{lin}})$ be the element in
group cohomology given by the $1$-cocycle in (1).
The element $\alpha_G$ is trivial
if and only if the natural action of $G$ on
$(\Z/p^s\Z)^{\oplus r}$ has a fixed point.
\end{enumerate}
\end{lem}

\begin{proof}
The first assertion can be checked by a direct calculation.
For the second assertion, $\alpha_G$ is trivial if and only if
there exists an element $w \in (\Z/p^s\Z)^{\oplus r}$ satisfying
$v = m(w) - w$ for every $g_{v,m} \in \AGL_r(\Z/p^s\Z)$.
This condition is equivalent
to the assertion that $-w$ is fixed by $G$
with respect to the natural action on $(\Z/p^s\Z)^{\oplus r}$.
\end{proof}

\begin{lem}
\label{Lemma:SylowFixedPoints}
Let $G \subset \AGL_r(\Z/p^s\Z)$ be a subgroup,
and $G_p \subset G$ a Sylow $p$-subgroup.
Assume that the natural action of $G_p$ on $(\Z/p^s\Z)^{\oplus r}$ has a fixed point.
Then the natural action of $G$ on $(\Z/p^s\Z)^{\oplus r}$ has a fixed point.
\end{lem}

\begin{proof}
Since the multiplication by $[G : G_p]$ on
$(\Z/p^s\Z)^{\oplus r}_{\mathrm{lin}}$
is an isomorphism of $G$-modules,
the restriction map
\[ H^1(G, (\Z/p^s\Z)^{\oplus r}_{\mathrm{lin}}) \to H^1(G_p, (\Z/p^s\Z)^{\oplus r}_{\mathrm{lin}}) \]
is injective.
Hence the assertion follows from Lemma \ref{Lemma:AffineFixedPoints}.
\end{proof}

\begin{lem}
\label{Lemma:AffineTransformationInAppendix}
Let $G \subset \AGL_r(\Z/p^s\Z)$ be a subgroup.
Assume that the following conditions are satisfied:
\begin{itemize}
\item The Sylow $p$-subgroup of $\pi(G)$ is cyclic.
\item For every element $g \in G$ whose order is a power of $p$, the natural action of $g$ on $(\Z/p^s\Z)^{\oplus r}$ has a fixed point.
\end{itemize}
Then the natural action of $G$ on $(\Z/p^s\Z)^{\oplus r}$ has a fixed point.
\end{lem}

\begin{proof}
We shall first show the restriction of $\pi$ to $G$ is injective.
To see this, take a non-trivial element
$g \in G \cap (\Ker \pi)$.
The natural action of $g$ on $(\Z/p^s\Z)^{\oplus r}$ is given by the translation by a non-zero vector.
Hence the order of $g$ is a power of $p$,
which contradicts the second assumption.
Therefore, the restriction of $\pi$ to $G$ is injective.

Let $G_p \subset G$ be a Sylow $p$-subgroup.
By the injectivity of the restriction of $\pi$ to $G_p$
and the first assumption, we see that $G_p$ is cyclic.
By the second assumption, the natural action of $G_p$ on $(\Z/p^s\Z)^{\oplus r}$ has a fixed point.
By Lemma \ref{Lemma:SylowFixedPoints},
the natural action of $G$ on $(\Z/p^s\Z)^{\oplus r}$ has a fixed point.
\end{proof}

\begin{proof}[\textit{\textbf{Proof of Proposition \ref{Proposition:AffineTransformation} (Serre)}}]
By Lemma \ref{Lemma:AffineTransformationInAppendix},
it is enough to show that the Sylow $p$-subgroup of $\pi(G)$ is cyclic.

If $r = 1$, the Sylow $p$-subgroup of $(\Z/p^s \Z)^{\times}$ is cyclic. (Recall that we are assuming $s \leq 2$ when $p = 2$.)

If $r = 2$, the Sylow $p$-subgroup of $\GL_2(\F_p)$ is cyclic (of order $p$).
\end{proof}

\iffalse
\begin{lem}
\label{Lemma:SubgroupIndex}
Let $G$ be a finite group, and $H, K \subset G$ subgroups.
Assume that $|H \cap K| = |H \cap g K g^{-1}|$ for every $g \in G$.
Then $[H : H \cap K]$ divides $[G : K]$.
\end{lem}

\begin{proof}
Consider the action of $H$ on $G/K$ from the left.
For any $gK \in G/K$, the size of the $H$-orbit of $gK$
is identified with $H/(H \cap g K g^{-1})$.
It has size
$[H : H \cap g K g^{-1}] = [H : H \cap K]$
by assumption.
Since the $H$-orbits on $G/K$ have the same size,
$[G : K]$ is divisible by $[H : H \cap K]$.
\end{proof}
\fi

The following result is a key to prove
the Hasse principle allowing exceptional sets of positive density.

\begin{prop}
\label{Proposition:AffineTransformation2}
Assume that $(p,r,s)$ satisfies one of Condition \ref{Condition:prs}.
Let $G \subset \AGL_r(\Z/p^s \Z)$ be a subgroup.
If $d_G > 0$, then $d_G \geq p^{-rs}$.
\end{prop}

If $p = 2$, it is easy to verify (using \texttt{GAP})
the assertion of Proposition \ref{Proposition:AffineTransformation2}
holds for every subgroup of $\AGL_1(\F_2)$, $\AGL_1(\Z/4 \Z)$, $\AGL_2(\F_2)$, and $\AGL_3(\F_2)$.

Therefore, in the rest of this section, we may assume $p$ is odd.
We shall prove
Proposition \ref{Proposition:AffineTransformation2} for subgroups of $\AGL_1(\Z/p^s \Z)$ and $\AGL_2(\F_p)$.

\begin{proof}[\textit{\textbf{Proof of Proposition \ref{Proposition:AffineTransformation2} (for subgroups of $\AGL_1(\Z/p^s \Z)$)}}]
Let
\[ G \subset \AGL_1(\Z/p^s \Z) \]
be a subgroup with $d_G > 0$.

If $G \cap (\Ker \pi)$ is trivial,
the restriction of $\pi$ to $G$ is injective.
Hence we have
\[ d_G \geq \frac{1}{|G|} \geq \frac{1}{|(\Z/p^s \Z)^{\times}|}
   = \frac{1}{p^{s-1}(p-1)} > p^{-s}. \]

If $G \cap (\Ker \pi)$ is non-trivial,
we put
$t := |G \cap (\Ker \pi)| > 1$.
Then we have  $|G| = t \cdot |\pi(G)|$, and
there are at least $t - 1$
fixed point free elements in $G$.
Since $G \cap (\Ker \pi)$ is a $p$-group, we have $t \geq p$.
Hence we have
\[
  d_G \geq \frac{t - 1}{|G|}
  = \frac{t - 1}{t} \cdot \frac{1}{|\pi(G)|}
  \geq \frac{p - 1}{p} \cdot \frac{1}{|\pi(G)|}.
\]
Since $|\pi(G)| \leq |(\Z/p^s \Z)^{\times}| = (p-1) p^{s-1}$,
we have
\[
  d_G \geq \frac{p - 1}{p} \cdot \frac{1}{(p-1) p^{s-1}}
  = p^{-s}.
\]
\end{proof}

In the rest of this subsection,
we shall consider subgroups of $\AGL_2(\F_p)$.

\begin{lem}
\label{Lemma:LargeCentralizer}
Let $G \subset \AGL_2(\F_p)$ be a subgroup.
If there exists a fixed point free element $g \in G$ with
$|\mathrm{Cent}_{\AGL_2(\F_p)}(g)| \leq p^2$,
then we have $d_G \geq p^{-2}$.
\end{lem}

\begin{proof}
The assertion follows from the following inequalities:
\[
d_G \geq \frac{|\mathrm{Conj}_G(g)|}{|G|}
= \frac{1}{|\mathrm{Cent}_{G}(g)|}
\geq \frac{1}{|\mathrm{Cent}_{\AGL_2(\F_p)}(g)|}
\geq p^{-2}.
\]
(Here the conjugacy class of $g$ is denoted by $\mathrm{Conj}_G(g)$,
and the centralizer of $g$ in $G$ is denoted by $\mathrm{Cent}_{G}(g)$.)
\end{proof}

We omit the proof of the following lemma
because it can be checked by a straightforward calculation.

\begin{lem}
\label{Lemma:FixedPoint}
\begin{enumerate}
\item Let $m \in \GL_2(\F_p)$ be an element such that $1$ is not an eigenvalue of $m$.
For every $v \in \F_p^{\oplus 2}$, the element $g_{v,m}$ has a fixed point.

\item 
For $v \in \F_p^{\oplus 2}$,
the element $g_{v,I_2}$ is fixed point free if and only if $v \neq 0$.

\item Let $m \in \GL_2(\F_p)$ be a non-identity semisimple matrix
such that $1$ is an eigenvalue of $m$.
There exist $p(p-1)$ elements $v \in \F_p^{\oplus 2}$
such that $g_{v,m}$ is fixed point free.
For such an element $g_{v,m}$, we have
\[ |\mathrm{Cent}_{\AGL_2(\F_p)}(g_{v,m})| = p(p-1). \]

\item Let $m \in \GL_2(\F_p)$ be a  non-identity unipotent matrix.
There are $p(p-1)$ elements $v \in \F_p^{\oplus 2}$
such that $g_{v,m}$ is fixed point free.
For such an element $g_{v,m}$, we have
\[ |\mathrm{Cent}_{\AGL_2(\F_p)}(g_{v,m})| = p^2. \]
\end{enumerate}
\end{lem}

\begin{proof}[\textit{\textbf{Proof of Proposition \ref{Proposition:AffineTransformation2} (for subgroups of $\AGL_2(\F_p)$)}}]
Let $G \subset \AGL_2(\F_p)$ be a subgroup with $d_G > 0$.
We want to show $d_G \geq p^{-2}$.

Take a fixed point free element $g_{v,m} \in G$.
Then $1$ is an eigenvalue of $m$ by Lemma \ref{Lemma:FixedPoint} (1).

\begin{itemize}
\item If $m$ is a non-identity semisimple matrix, then
\[ |\mathrm{Cent}_{\AGL_2(\F_p)}(g_{v,m})| = p(p-1) \leq p^2 \]
by Lemma \ref{Lemma:FixedPoint} (3).
Hence $d_G \geq p^{-2}$ by Lemma \ref{Lemma:LargeCentralizer}. 

\item If $m$ is a non-identity unipotent matrix, then
\[ |\mathrm{Cent}_{\AGL_2(\F_p)}(g_{v,m})| = p^2 \]
by Lemma \ref{Lemma:FixedPoint} (4).
Hence $d_G \geq p^{-2}$ by Lemma \ref{Lemma:LargeCentralizer}. 
\end{itemize}

Therefore, by Lemma \ref{Lemma:FixedPoint} (2),
it remains to consider the case where
every fixed point free element of $G$
is of the form $g_{v,I_2}$
for some $v \neq 0$.
The element $g_{v,I_2} \ (v \neq 0)$ acts on $\F_p^{\oplus 2}$ as a non-trivial translation by $v$.
Since $d_G > 0$, there exists at least one such element in $G$.
In particular, the restriction $\pi|_G$ of $\pi \colon \AGL_2(\F_p) \to \GL_2(\F_p)$ to $G$ is
not injective.
Since $\Ker \pi \cong \F_p^{\oplus 2}$,
we have
\[
  |\Ker(\pi|_G)| = p^2
  \quad \text{or} \quad
  |\Ker(\pi|_G)| = p.
\]

First, we consider the case $|\Ker(\pi|_G)| = p^2$.
In this case, the action of $G$ on $\F_p^{\oplus 2}$ is transitive.
By a theorem of Cameron-Cohen \cite[p.136]{CameronCohen} (see also \cite[Theorem 5]{Serre:Jordan}),
we have $d_G \geq p^{-2}$.

It remains to consider the case $|\Ker(\pi|_G)| = p$.
In this case, after changing the coordinates if necessary,
we may assume
$g_{e_1,I_2} \in \Ker(\pi|_G)$.
Then we have 
\[
    \Ker(\pi|_G) =  \langle g_{e_1, I_2} \rangle,
\]
i.e. the group $\Ker(\pi|_G)$ is generated by $g_{e_1, I_2}$.
The $p - 1$ elements
$g_{a e_1, I_2} \ (1 \leq a \leq p - 1)$ are fixed point free.
Hence we have
$d_G \geq (p-1)/|G|$.
Therefore,
in order to prove $d_G \geq p^{-2}$,
it is sufficient to prove
\[
  |G| \leq (p - 1) p^2.
\]

Let $B_0 \subset \GL_2(\F_p)$ (resp.\ $T_0 \subset \GL_2(\F_p)$) be the subgroup of upper triangular (resp.\ diagonal) matrices.
We put
\[
  B := \pi^{-1}(B_0) = \F_p^{\oplus 2} \rtimes B_0,
  \quad
  T := \pi^{-1}(T_0) = \F_p^{\oplus 2} \rtimes T_0.
\]
Then $B$ is the normalizer of $\langle g_{e_1, I_2} \rangle$
in $\AGL_2(\F_p)$.
Since $\langle g_{e_1, I_2} \rangle = \Ker(\pi|_G)$
is a normal subgroup of $G$,
we see that $G$ is a subgroup of $B$.
We consider the following homomorphism
\[
  \chi \colon T_0 \to \F_p^{\times},
  \quad
  \begin{pmatrix} b_1 & 0 \\ 0 & b_2 \end{pmatrix} \mapsto b_1.
\]
Let $H$ be the kernel of the map
\[ G \cap T \to \F_p^{\times},\quad g_{v,m} \mapsto \chi(m).
\]
We have the following sequence of subgroups of $G$:
\[
G \ \supset \ (G \cap T)
\ \supset \ H
\ \supset \ \Ker(\pi|_G) = \langle g_{e_1, I_2} \rangle.
\]

Since $G$ is a subgroup of $B$, we have
\[
  [G : G \cap T] \leq [B : T] = [B_0 : T_0] = p.
\]
Since $H$ is the kernel of the homomorphism $G \cap T \to \F_p^{\times}$,
we have
\[
  [G \cap T : H] \leq |\F_p^{\times}| = p - 1.
\]
Since $|\langle g_{e_1, I_2} \rangle| = p$,
in order to prove $|G| \leq (p-1) p^2$,
it is sufficient to prove
$H = \langle g_{e_1, I_2} \rangle$.

We shall prove $H = \langle g_{e_1, I_2} \rangle$.
Take an element $g_{v',m'} \in H$.
We put
\[
  v' = a'_1 e_1 + a'_2 e_2,
  \quad
  m' = \begin{pmatrix} 1 & 0 \\ 0 & b'_2 \end{pmatrix}
\]
for some $a'_1, a'_2 \in \F_p$ and $b'_2 \in \F_p^{\times}$.

If $m' = I_2$, then $g_{v',I_2} \in \Ker(\pi|_G) = \langle g_{e_1, I_2} \rangle$.

Assume $m' \neq I_2$.
We shall show this case cannot happen.
\begin{itemize}
\item 
If $a'_1 \neq 0$, then
$g_{v',m'}$ is fixed point free.
This cannot happen because we are assuming every fixed point free element of $G$
is of the form $g_{v,I_2}$
for some $v \neq 0$.

\item
If $a'_1 = 0$, we consider the product
\[
  g_{e_1,I_2} \cdot g_{v',m'} = g_{e_1+v', m'} \in H.
\]
Since
$e_1 + v' = e_1 + a'_2 e_2$,
this cannot happen by the same reason as above.
\end{itemize}
\end{proof}

% \section{Examples and counterexamples}
% \label{Section:Example}

\subsection{Examples of subgroups of $\AGL_r(\Z/p^s \Z)$}

Here we give subgroups of the following groups
\begin{itemize}
\item $\AGL_1(\Z/2^s\Z)$ for $s \geq 3$,
\item $\AGL_2(\Z/p^s\Z)$ for $s \geq 2$,
\item $\AGL_r(\F_p)$ for $p \geq 3$ and $r \geq 3$,
\item $\AGL_r(\F_2)$ for $r \geq 4$
\end{itemize}
such that $d_G = 0$, but
the natural action of $G$ on $(\Z/p^s \Z)^{\oplus r}$
has no fixed point.
Hence the assumptions on the triple $(p,r,s)$
in Proposition \ref{Proposition:AffineTransformation}
cannot be weakened.

\begin{ex}
\label{Example:HassePrincipleFail}
\begin{enumerate}
\item Assume that $s \geq 3$.
Then we put
\[ G_1 := \{ (0,1),\ (0,-1),\ (2^{s-1}, 1 + 2^{s-1}),\ (2^{s-1}, - 1 + 2^{s-1}) \}. \]
It is a subgroup of $\AGL_1(\Z/2^s\Z)$ for $s \geq 3$.

\item Assume that $s \geq 2$.
For $a, b \in \Z/p^s\Z$, we define
$m'_{a,b} \in \GL_2(\Z/p^s\Z)$ by
\begin{align*}
  m'_{a,b}(e_1) &= (1 + p^{s-1} a)e_1 + p^{s-1} b e_2, \\
  m'_{a,b}(e_2) &= (1 + p^{s-1} a)e_2.
\end{align*}
We put
$G_2 := \{ (p^{s-1} b e_1, m'_{a,b}) \}_{a,b \in \Z/p\Z}$.
It is a subgroup of $\AGL_2(\Z/p^s\Z)$
for $s \geq 2$.

\item Assume $p \geq 3$ and $r \geq 3$.
For $a,b \in \F_p$, we define
$h_{a,b} \in \GL_r(\F_{p})$ by
\begin{align*}
    h_{a,b}(e_i) &= e_i \quad (i \neq 2,3), \\
    h_{a,b}(e_2) &= ae_1 + e_2, \\
    h_{a,b}(e_3) &= (b + a(a+1)/2) e_1 + ae_2 + e_3.
\end{align*}
We put $G_3: = \{ (be_1, h_{a,b}) \}_{a,b \in \F_p}$.
It is a subgroup of $\AGL_r(\F_p)$
for $p \geq 3$ and $r \geq 3$.

\item Assume $p=2$ and $r \geq 4$.
We define $h'_1, h'_2 \in \GL_r(\F_2)$ by
\begin{align*}
    h'_1(e_i) &= e_i \quad (i \neq 1), \\
    h'_1(e_1) &= e_1 + e_3, \\
    h'_2(e_i) &= e_i \quad (i \neq 1,2), \\
    h'_2(e_1) &= e_1 + e_4, \\
    h'_2(e_2) &= e_2 + e_3.
\end{align*}
We put
$G_4 := \{ (0,I_r),\,(0,h'_1),\,(e_4,h'_2),\,(e_4,h'_1h'_2) \}$.
It is a subgroup of $\AGL_r(\F_2)$ for $r \geq 4$.
\end{enumerate}
It is easy to see that, for every $i \in \{ 1,2,3,4 \}$,
we have $d_{G_i} = 0$, but there does not exist
an element of $(\Z/p^s\Z)^{\oplus r}$ fixed by $G_i$. 
Note that all the subgroups $G_i$ constructed above are solvable.
\end{ex}

Next, for any triple $(p,r,s)$ as in
Proposition \ref{Proposition:AffineTransformation2},
we give a subgroup
\[ G \subset \AGL_r(\Z/p^s \Z) \]
with $d_G = p^{-rs}$.
Hence the lower bound ``$d_G  \geq p^{-rs}$''
in Proposition \ref{Proposition:AffineTransformation2}
is optimal.

\begin{ex}
\label{Example:SmallProportion}
Let $p$ be a prime number, and $r,s \geq 1$ positive integers.
For $a_1,\ldots,a_r \in \Z/p^s \Z$,
we define a square matrix $b_{a_1,\ldots,a_r}$ of size $r$ by
\[
(b_{a_1,\ldots,a_r})_{ij}
:=
\begin{cases}
a_j & i = 1 \\
\delta_{ij} & 2 \leq j \leq r,
\end{cases}
\]
where $\delta_{ij}$ is Kronecker's delta.
Let $G \subset \AGL_r(\Z/p^s \Z)$ be the subgroup generated by the following elements:
\begin{itemize}
\item $(t e_1, I_r)$ for $t \in p^{s-1}\Z/p^s \Z$.
\item $(0, b_{a_1,\ldots,a_r})$ for
$a_1 \in (\Z/p^s \Z)^{\times}$ and
$a_2,\ldots,a_r \in \Z/p^s \Z$.
\end{itemize}
Then, $G$ is a solvable group of order $(p-1)p^{rs}$.
The $p-1$ elements of the form $(t e_1, I_r)$ for $t \neq 0$
are fixed point free.
Any other element has a fixed point on
$(\Z/p^s\Z)^{\oplus r}$.
Hence we have $d_G = p^{-rs}$.
\end{ex}

\begin{rem}
There are other types subgroups
$G \subset \AGL_r(\Z/p^s \Z)$ with $d_G = p^{-rs}$.
Here we give another example when $s = 1$.
Fix an embedding
$\varphi \colon \F_{p^r}^{\times} \hookrightarrow \GL_r(\F_p)$.
Then
$G' := \F_p^{\oplus r} \rtimes \varphi(\F_{p^r}^{\times})$
is a subgroup of $\AGL_r(\F_p)$ with $d_{G'} = p^{-r}$.
\end{rem}

\begin{rem}
\label{Remark:p=2}
Here we give another proof of
Proposition \ref{Proposition:AffineTransformation}
when $(p,r,s) = (2,3,1)$.
The following proof does not require computer verification.
Let $G \subset \AGL_3(\F_2)$ be a subgroup.
By Lemma \ref{Lemma:AffineTransformationInAppendix},
we may assume $G$ is a $2$-group,
and $\pi(G)$ is contained in the subgroup of upper
triangular matrices in $\GL_3(\F_2)$.
If $\pi(G)$ is cyclic, the assertion follows by Lemma \ref{Lemma:AffineTransformationInAppendix}.
If $\pi(G)$ is not cyclic,
it can be checked by hand that
$\pi(G)$ is equal to one of the following groups
\begin{align*}
G_1 &:= \{\, u_{a,b,c} \mid a,b,c \in \F_2 \,\}, \\
G_2 &:= \{\, u_{a,b,0} \mid a,b \in \F_2 \,\}, \\
G_3 &:= \{\, u_{0,b,c} \mid b,c \in \F_2 \,\}.
\end{align*}
Here we put
$u_{a,b,c} := \begin{pmatrix} 1 & a & b \\ 0 & 1 & c \\ 0 & 0 & 1 \end{pmatrix}$.
It can be checked directly that
if $\pi(G)$ is equal to one of $G_1,G_2,G_3$,
the natural action of $G$ on $\F_2^{\oplus 3}$ has a fixed point.
\end{rem}

\section{Proof of Theorem \ref{Theorem:GaloisCohomology}}
\label{Section:ProofMainTheorem}

We first recall the statement of the Chebotarev density theorem.

Let $K$ be a global field.
For a finite place $v$ of $K$,
the (arithmetic) Frobenius element at $v$
is defined by $\Frob_v(x) := x^{|\kappa(v)|}$ for $x \in \kappa(v)^{\mathrm{sep}}$.
The kernel of a canonical homomorphism
$\Gal(K_v^{\mathrm{sep}}/K_v) \to \Gal(\kappa(v)^{\mathrm{sep}}/\kappa(v))$
is called the \emph{inertia group} at $v$,
which is denoted by $I_v$.

Let $G$ be a finite group, and 
$\rho \colon \Gal(K^{\mathrm{sep}}/K) \to G$
a continuous homomorphism.
For a finite place $v$ of $K$,
we say $\rho$ is \emph{unramified} at $v$ if
the restriction of $\rho$ to $I_v$ is trivial.
It is satisfied for all but finitely many $v$.
If $\rho$ is unramified at $v$, we choose
an embedding
$\iota_v \colon K^{\mathrm{sep}} \hookrightarrow K_v^{\mathrm{sep}}$
and a lift
$\widetilde{\Frob_v} \in \Gal(K_v^{\mathrm{sep}}/K_v)$
of $\Frob_v$.
We consider it as an element of $\Gal(K^{\mathrm{sep}}/K)$
via the embedding
$\Gal(K_v^{\mathrm{sep}}/K_v) \hookrightarrow \Gal(K^{\mathrm{sep}}/K)$
induced by $\iota_v$.
We put
$\rho(\Frob_v) := \rho(\widetilde{\Frob_v}) \in G$.
The conjugacy class of $\rho(\Frob_v)$ does not depend on
the choices.

\begin{thm}[Chebotarev density theorem]
\label{Theorem:Chebotarev}
Let $K$ be a global field, $G$ a finite group, and
$\rho \colon \Gal(K^{\mathrm{sep}}/K) \to G$
a continuous homomorphism.
Let $S_g$ be the set of finite places $v$ of $K$ such that
$\rho$ is unramified at $v$ and $\rho(\Frob_v)$ belongs to
the conjugacy class $\mathrm{Conj}_G(g)$.
Then the Dirichlet density $\delta(S_g)$ exists and is equal to
$|\mathrm{Conj}_G(g)|/|G|$.
\end{thm}

\begin{proof}
See \cite[Chapter 6, Theorem 6.3.1]{FriedJarden} for example.
\end{proof}

\begin{rem}
\label{Remark:density}
If $K$ is a number field, the same assertion holds for the natural density instead of the Dirichlet density.
On the other hand, the natural density does not work well in the function field case;
see \cite[Remark 3.3]{Ballot}.
\end{rem}

Now we shall prove Theorem \ref{Theorem:GaloisCohomology}.
Let notation be as in  Theorem \ref{Theorem:GaloisCohomology}.
We fix an isomorphism $M \cong (\Z/p^s \Z)^{\oplus r}$.
Let
$\alpha \in H^1(\Gal(K^{\mathrm{sep}}/K), M)$
be an element in the kernel of $\Res_{M,S}$.
We shall show $\alpha$ is trivial.
As explained in the introduction, the element $\alpha$ corresponds
to a continuous homomorphism 
\[ \psi \colon \Gal(K^{\mathrm{sep}}/K) \to \AGL_r(\Z/p^s \Z). \]
Let $G := \psi(\Gal(K^{\mathrm{sep}}/K))$ be the image of $\psi$.
We put
\[ T := \{\, g \in G \mid g \ \text{is fixed point free.} \,\}. \]

For a finite place $v$ where $\psi$ is unramified,
the conjugacy class of $\psi(\Frob_v) \in G$
is well-defined.
The restriction of $\alpha$ to
$H^1(\Gal(K_v^{\mathrm{sep}}/K_v), M)$ is trivial
if and only if $\psi(\Frob_v) \notin T$.
In particular, we have $\psi(\Frob_v) \notin T$ for every $v \in S$.

By Theorem \ref{Theorem:Chebotarev},
we have
\[ d_G = |T|/|G| < p^{-rs}. \]
By Proposition \ref{Proposition:AffineTransformation2},
we have $d_G = 0$.
By Proposition \ref{Proposition:AffineTransformation},
there exists an element of $(\Z/p^s \Z)^{\oplus r}$ fixed by
every element of $G$.
Therefore, by Lemma \ref{Lemma:AffineFixedPoints},
the element
$\alpha \in H^1(\Gal(K^{\mathrm{sep}}/K), M)$
is trivial.

The proof of Theorem \ref{Theorem:GaloisCohomology} is complete.
\hfill $\square$

\section{Proof of Theorem \ref{Theorem:Counterexample}}
\label{Section:Counterexample}

\subsection{Some results from the Inverse Galois Theory}

In order to construct counterexamples of the Hasse principle satisfying the conditions in Theorem \ref{Theorem:Counterexample},
we shall use the following result from the Inverse Galois Theory.

\begin{thm}[Shafarevich, Sonn]
\label{Theorem:ShafarevichSonn}
Let $K$ be a global field (of any characteristic),
and $G$ a finite solvable group.
Then there exists a finite Galois extension $L/K$ with
\[ \Gal(L/K) \cong G \]
such that every decomposition group of $\Gal(L/K)$ is cyclic.
\end{thm}

\begin{proof}
Shafarevich proved the existence of a finite Galois extension $L/K$ with
\[ \Gal(L/K) \cong G. \]
Sonn observed that the extension $L/K$
constructed by Shafarevich
satisfies the condition that
every decomposition group of $\Gal(L/K)$ is cyclic.
It is proved in the proof of \cite[Theorem 2]{Sonn}.
It is used by Jahnel--Loughran;
see \cite[Lemma 2.7]{JahnelLoughran}.
\end{proof}

\subsection{Proof of Theorem \ref{Theorem:Counterexample} (1)}

Take a subgroup $G_i \subset \AGL_r(\Z/p^s \Z)$
as in Example \ref{Example:HassePrincipleFail}.
Then $G_i$ is a solvable group such that
$d_{G_i} = 0$, but
the natural action of $G_i$ on $(\Z/p^s \Z)^{\oplus r}$
has no fixed point.

Since $G_i$ is solvable,
by Theorem \ref{Theorem:ShafarevichSonn},
we can find a finite Galois extension $L/K$ with
$\Gal(L/K) \cong G_i$
such that every decomposition group of $\Gal(L/K)$ is cyclic.

Consider the continuous homomorphism
\[ \psi \colon \Gal(K^{\mathrm{sep}}/K) \to \Gal(L/K) \cong G_i \hookrightarrow \AGL_r(\Z/p^s \Z). \]
The composition of $\psi$ with the projection
\[
  \pi \colon \AGL_r(\Z/p^s \Z) \to \GL_r(\Z/p^s \Z)
\]
gives a linear action of $\Gal(K^{\mathrm{sep}}/K)$ on
$(\Z/p^s \Z)^{\oplus r}$.
It gives a Galois module. We denote it by $M$.

The composition of $\psi$ with the first projection
\[
  \AGL_r(\Z/p^s \Z) \to (\Z/p^s \Z)^{\oplus r} \cong M
\]
is a $1$-cocycle, which gives an element
$\alpha_{M} \in H^1(\Gal(K^{\mathrm{sep}}/K), M)$.
Since the natural action of $G_i$ on $(\Z/p^s \Z)^{\oplus r}$
has no fixed point,
the element $\alpha_{M}$ is non-trivial
by Lemma \ref{Lemma:AffineFixedPoints} (2).

Let $v$ be a place of $K$.
The image of $\Gal(K_v^{\mathrm{sep}}/K_v)$ in $\Gal(L/K) \cong G_i$
is the decomposition group at $v$.
It is a cyclic subgroup of $G_i$ by our choice of $L/K$.
Since $d_{G_i} = 0$, the natural action of
$\Gal(K_v^{\mathrm{sep}}/K_v)$ on $(\Z/p^s \Z)^{\oplus r}$
has a fixed point.
By Lemma \ref{Lemma:AffineFixedPoints} (2),
the element $\alpha_M$ sits in the kernel of the map
\[ H^1(\Gal(K^{\mathrm{sep}}/K), M) \to H^1(\Gal(K_v^{\mathrm{sep}}/K_v), M). \]
Since $v$ is arbitrary, we have $\Res_{M}(\alpha_{M}) = 0$.
Hence $\Res_{M}$ is not injective.
\hfill $\square$

\subsection{Proof of Theorem \ref{Theorem:Counterexample} (2)}

Take a solvable subgroup $G \subset \AGL_r(\Z/p^s \Z)$
with $d_G = p^{-rs}$
in Example \ref{Example:SmallProportion}.

Similarly as in (1),
we take a finite Galois extension $L/K$ with $\Gal(L/K) \cong G$.
Let $M$ be the $\Gal(K^{\mathrm{sep}}/K)$-module whose underlying
abelian group is $(\Z/p^s \Z)^{\oplus r}$
such that $\Gal(K^{\mathrm{sep}}/K)$ acts on it via
\[ \Gal(K^{\mathrm{sep}}/K) \to \Gal(L/K) \cong G \hookrightarrow \AGL_r(\Z/p^s \Z) \overset{\pi}{\to} \GL_r(\Z/p^s \Z) \cong \Aut(M). \]
The composition of the following maps
\[ \Gal(K^{\mathrm{sep}}/K) \to \Gal(L/K) \cong G \hookrightarrow \AGL_r(\Z/p^s \Z) \to (\Z/p^s \Z)^{\oplus r} \cong M \]
is a $1$-cocycle.
It gives an element
$\alpha_M \in H^1(\Gal(K^{\mathrm{sep}}/K), M)$.

Let $S$ be the set of finite places $v$
such that $\psi$ is unramified at $v$
and $\psi(\Frob_v)$ has a fixed point on
$(\Z/p^s \Z)^{\oplus r}$.
By Theorem \ref{Theorem:Chebotarev}, the Dirichlet density $\delta(S)$ exists and
\[ \delta(S) = 1 - d_G = 1 - p^{-rs}. \]

Then, the element $\alpha_M$ is non-trivial,
and it sits in the kernel of the map
\[ H^1(\Gal(K^{\mathrm{sep}}/K), M) \to H^1(\Gal(K_v^{\mathrm{sep}}/K_v), M) \]
for every $v \in S$.

Hence the map $\Res_{M,S}$ is not injective.
\hfill $\square$

\section{The Hasse principle for flexes on cubic curves}
\label{Section:Flex}

\subsection{The statement of the theorem for flexes}

Let $C \subset \PP^2$ be a smooth cubic curve over a field $K$.
A point $P$ on $C$ is called a \emph{flex}
(or an \emph{inflection point})
if the tangent line at $P$ intersects with $C$ with multiplicity $3$.
Since the work of Hesse in the 19th century,
flexes on smooth cubic curves have been studied by many mathematicians
\cite[Section 2]{ArtebaniDolgachev}, \cite[Chapter 3]{Dolgachev:ClassicalAlgebraicGeometry}.
It is well-known that, if $\chara K \neq 3$,
any smooth cubic curve over $K$ has exactly $9$ flexes over $K^{\mathrm{alg}}$.
The number of $K$-rational flexes is usually less than $9$.

In this section,
as an application of Theorem \ref{Theorem:GaloisCohomology},
we shall prove the following theorem.

\begin{thm}
\label{Theorem:Flex}
Let $K$ be a global field with $\chara K \neq 3$ (resp.\ $\chara K = 3$).
Let $C \subset \PP^2$ be a smooth cubic curve over $K$.
Assume that there is a set $S$ of finite places of $K$ such that
\begin{itemize}
\item the Dirichlet density $\delta(S)$ exists and satisfies $\delta(S) > 8/9$
(resp.\ $\delta(S) > 2/3$), and
\item $C$ has a $K_v$-rational flex for every $v \in S$.
\end{itemize}
Then $C$ has a $K$-rational flex.
\end{thm}

\begin{rem}
When $\chara K \neq 3$,
the Hasse principle for flexes follows easily from the results in the literature
(e.g., \cite[Chapter I, Theorem 4.10 (a)]{Milne}, \cite[Chapter I, Corollary 9.6]{Milne}).
\end{rem}

\subsection{Proof of Theorem \ref{Theorem:Flex} --- The case $\chara K \neq 3$}

We first consider the case $\chara K \neq 3$.

Let $K$ be a field with $\chara K \neq 3$.
Let $K^{\mathrm{alg}}$ be an algebraic closure of $K$,
and $K^{\mathrm{sep}}$ a separable closure of $K$ inside $K^{\mathrm{alg}}$.
Let $C \subset \PP^2$ be a smooth cubic curve over $K$.
Let
$\mathrm{He}(C) \subset \PP^2$
be the \emph{Hessian curve}.
(It is a cubic curve over $K$. See Section \ref{Section:FlexHessian}.)
The scheme of flexes on $C$ is defined by
\[ \mathrm{Flex}_C := C \cap \mathrm{He}(C). \]

\begin{lem}
\label{Lemma:Flex}
\begin{enumerate}
\item The scheme $\mathrm{Flex}_C$ is finite and \'etale over $K$.
There exist exactly $9$ flexes on $C$ over $K^{\mathrm{alg}}$,
and all of them are defined over $K^{\mathrm{sep}}$.

\item There is $\Gal(K^{\mathrm{sep}}/K)$-equivariant simply transitive action
\[ \eta \colon \Jac(C)[3](K^{\mathrm{sep}}) \times \mathrm{Flex}_C(K^{\mathrm{sep}}) \to \mathrm{Flex}_C(K^{\mathrm{sep}}). \]
\end{enumerate}
\end{lem}

\begin{proof}
Take a flex $P_0 \in \mathrm{Flex}_C(K^{\mathrm{alg}})$
over $K^{\mathrm{alg}}$;
see Proposition \ref{Proposition:Flex:CharNot2} and Proposition \ref{Proposition:Flex:Char2}.
The tangent line of $C$ at $P_0$ is defined by
$G_{P_0} = 0$
for a linear form $G_{P_0} \in K^{\mathrm{alg}}[X,Y,Z]$.

We shall construct a bijection
$\eta_{P_0} \colon \mathrm{Flex}_C(K^{\mathrm{alg}}) \to \Jac(C)[3](K^{\mathrm{alg}})$
as follows:
\begin{itemize}
\item For a flex $P \in \mathrm{Flex}_C(K^{\mathrm{alg}})$,
let $G_P = 0$ be the tangent line at $P$.
Since $\divi(G_P/G_{P_0}) = 3P - 3P_0$,
the divisor class of $P - P_0$ is killed by $3$.
We put $\eta_{P_0}(P) := [P - P_0]$.

\item Conversely, every element of $\Jac(C)[3](K^{\mathrm{alg}})$
is represented by $[P - P_0]$ for some $P \in C(K^{\mathrm{alg}})$
such that $3P - 3P_0$ is linearly equivalent to $0$; see \cite[Chapter IV, Theorem 4.11]{Hartshorne}.
Take a rational function $G_1/G_2$ with
$\divi(G_1/G_2) = 3P - 3P_0$.
The divisor cut out by the line $G_1 = 0$ is $3P$.
Hence $P$ is a flex.
\end{itemize}

There are exactly $9$ distinct $K^{\mathrm{alg}}$-rational flexes.
Since both $C$ and $\mathrm{He}(C)$ are cubic curves,
the multiplicity at every geometric point of
$C \cap \mathrm{He}(C)$ is equal to $1$
by \cite[Chapter V, Example 1.4.2]{Hartshorne}.
Thus, $\mathrm{Flex}_C$ is \'etale over $K$.
Hence we have
$\mathrm{Flex}_C(K^{\mathrm{alg}}) = \mathrm{Flex}_C(K^{\mathrm{sep}})$.
Then, the map $\eta$ defined by
$(\alpha, Q) \mapsto \eta_{P_0}^{-1}(\eta_{P_0}(Q) + \alpha)$
does not depend on the choice of $P_0$.
It gives a simply transitive action of 
$\Jac(C)[3](K^{\mathrm{sep}})$ on $\mathrm{Flex}_C(K^{\mathrm{sep}})$.
\end{proof}

We shall prove Theorem \ref{Theorem:Flex} when $\chara K \neq 3$.
We put $M := \Jac(C)[3](K^{\mathrm{sep}})$.
By Lemma \ref{Lemma:Flex} (2), we obtain
an element of $H^1(\Gal(K^{\mathrm{sep}}/K), M)$ as follows.
Fix an element $x_0 \in \mathrm{Flex}_C(K^{\mathrm{sep}})$
and identify $M$ and $\mathrm{Flex}_C(K^{\mathrm{sep}})$
via the map
$y \mapsto \eta_{x_0}(y) := \eta(y,x_0)$.

For an element $y \in M$,
we consider the map
\[
  \varphi_y \colon \Gal(K^{\mathrm{sep}}/K) \to M,
  \quad \varphi_y(\sigma) := \eta_{x_0}^{-1}(\sigma \cdot \eta_{x_0}(y)) - y.
\]
It is easy to see that $\varphi_y$ does not depend on the choice of $y$,
and it satisfies the cocycle condition.
Hence we have an element
$\alpha_{C} = [\varphi_y] \in H^1(\Gal(K^{\mathrm{sep}}/K), M)$.
The element $\alpha_{C}$ is trivial if and only if $C$ has a $K$-rational flex.
The same assertion holds over $K_v$ for every $v$.
Thus, $\alpha_{C}$ sits in the kernel of $\Res_{M, S}$.

By Theorem \ref{Theorem:GaloisCohomology},
the map $\Res_{M, S}$ is injective.
Consequently, $\alpha_{c}$ is trivial, and $C$ has a $K$-rational flex.

The proof of Theorem \ref{Theorem:Flex} is complete
when $\chara K \neq 3$.
\hfill $\square$

\subsection{Proof of Theorem \ref{Theorem:Flex} --- The case $\chara K = 3$}

Next, we shall consider the case $\chara K = 3$.
Although the scheme of flexes is not \'etale in characteristic $3$,
the Hasse principle for flexes holds in this case.

We first recall basic results on the flexes on smooth cubic curves in characteristic $3$.
Assume $\chara K = 3$.
Let $C \subset \PP^2$ be a smooth cubic curve over $K$,
$\mathrm{He}(C) \subset \PP^2$
the \emph{Hessian curve} of $C$, and
\[ \mathrm{Flex}_C := C \cap \mathrm{He}(C) \]
the scheme of flexes on $C$;
see Section \ref{Section:FlexHessian}.
Then, $\Jac(C)[3](K^{\mathrm{alg}})$ is
isomorphic to either $\F_3$ or $0$.
In the former case, we say $\Jac(C)$ is \emph{ordinary}.
Otherwise, we say it is \emph{supersingular}.

We omit the proof of the following lemma
because it is similar to Lemma \ref{Lemma:Flex}.

\begin{lem}
\label{Lemma:Flex:Characteristic3}
\begin{enumerate}
\item There is a map
\[ \eta \colon \Jac(C)[3](K^{\mathrm{alg}}) \times \mathrm{Flex}_C(K^{\mathrm{alg}}) \to \mathrm{Flex}_C(K^{\mathrm{alg}}) \]
which gives a simply transitive action of
$\Jac(C)[3](K^{\mathrm{alg}})$ on $\mathrm{Flex}_C(K^{\mathrm{alg}})$.

\item If $\Jac(C)$ is ordinary, $C$ has exactly $3$ flexes over $K^{\mathrm{alg}}$.
Otherwise, $C$ has a unique flex over $K^{\mathrm{alg}}$.

\item If $C$ has at least one $K^{\mathrm{sep}}$-rational flex,
the number of $K^{\mathrm{sep}}$-rational flexes is equal to
the number of elements in $\Jac(C)[3](K^{\mathrm{sep}})$.
\end{enumerate}
\end{lem}

In characteristic $3$,
a $K^{\mathrm{alg}}$-rational flex which is stable under $\Aut(K^{\mathrm{alg}}/K)$
might not be defined over $K$
due to the possible inseparability of the field of definition.
To treat this issue,
we shall use Greenberg's approximation theorem,
which is a special case of Artin's approximation theorem.

\begin{lem}
\label{Lemma:Approximation}
Let $K$ be a global field with $\chara K > 0$, and $v$ a finite place of $K$.
We fix an embedding $K^{\mathrm{alg}} \hookrightarrow K_v^{\mathrm{alg}}$.
For a finite $K$-scheme $T$,
every $K_v$-rational point on $T$ is defined over
a finite separable extension of $K$.
\end{lem}

\begin{proof}
We give a brief sketch of the proof because
essentially the same result was proved
in the proof of
\cite[Theorem 4.1]{IshitsukaIto:LocalGlobalChar2}.
Let $R := \mathcal{O}_{K,v} \subset K$ be the discrete valuation ring
corresponding to $v$,
and $R^{\mathrm{h}}$ the Henselization of $R$.
Since $R$ is an excellent discrete valuation ring,
by Greenberg's approximation theorem \cite[Theorem 1]{Greenberg:RationalPoints},
every $K_v$-rational point on $T$ is approximated by
a $\Frac(R^{\mathrm{h}})$-rational point.
Since $T$ is finite over $K$ and the extension $\Frac(R^{\mathrm{h}})/K$ is separable and algebraic,
every $K_v$-rational point of $T$ is defined over $K^{\mathrm{sep}}$.
\end{proof}

Now we shall prove Theorem \ref{Theorem:Flex} when $\chara K = 3$.

We first consider the ordinary case.
Assume that $\Jac(C)$ is ordinary.
Since the set $S$ is non-empty,
the cubic curve $C$ has a $K_v$-rational flex for a finite place $v$.
Thus the scheme $\mathrm{Flex}_C$ has a $K_v$-rational point.
By Lemma \ref{Lemma:Approximation},
it has a $K^{\mathrm{sep}}$-rational point,
and $C$ has a $K^{\mathrm{sep}}$-rational flex.
Since $\Jac(C)$ is ordinary,
$\Jac(C)[3](K^{\mathrm{sep}})$ is isomorphic to either $\F_3$ or $0$.
By Lemma \ref{Lemma:Flex:Characteristic3} (3),
the number of $K^{\mathrm{sep}}$-rational flexes on $C$
is either $3$ or $1$.
\begin{itemize}
\item
If $C$ has 3 flexes over $K^{\mathrm{sep}}$,
$\Gal(K^{\mathrm{sep}}/K)$ acts on $\mathrm{Flex}_C(K^{\mathrm{sep}})$
via $\AGL_1(\F_3)$.
Since $\delta(S) > 2/3$,
Theorem \ref{Theorem:Chebotarev},
Proposition \ref{Proposition:AffineTransformation},
and Proposition \ref{Proposition:AffineTransformation2}
imply the action of $\Gal(K^{\mathrm{sep}}/K)$ on $\mathrm{Flex}_C(K^{\mathrm{sep}})$
has a fixed point.
It corresponds to a $K$-rational flex.

\item 
If $C$ has a unique $K^{\mathrm{sep}}$-rational flex,
it is fixed by $\Gal(K^{\mathrm{sep}}/K)$ by the uniqueness.
Hence $C$ has a $K$-rational flex.
\end{itemize}

Next, we consider the supersingular case.
Assume that $\Jac(C)$ is supersingular.
Since $C$ has a $k_{v_0}$-rational flex,
by Lemma \ref{Lemma:Approximation},
$C$ has at least one $K^{\mathrm{sep}}$-rational flex.
Since $\Jac(C)[3](K^{\mathrm{sep}}) = 0$,
$C$ has a unique $K^{\mathrm{sep}}$-rational flex
by Lemma \ref{Lemma:Flex:Characteristic3} (3).
It is fixed by $\Gal(K^{\mathrm{sep}}/K)$ by the uniqueness.
Therefore, $C$ has a $K$-rational flex.

The proof of Theorem \ref{Theorem:Flex} is complete
when $\chara K = 3$.
\hfill $\square$

\begin{rem}
Similar arguments appeared in
\cite{IshitsukaIto:LocalGlobal} and \cite{IshitsukaIto:LocalGlobalChar2}.
See the proof of \cite[Theorem 5.1 (2)]{IshitsukaIto:LocalGlobal}
for the ordinary case,
and the proof of \cite[Theorem 4.1]{IshitsukaIto:LocalGlobalChar2}
for the supersingular case.
\end{rem}

\begin{rem}
When $\chara K = 3$, the lower bound $2/3$ in Theorem \ref{Theorem:Flex}
is optimal.
Let $C$ be a smooth cubic curve over $K$
such that $\Jac(C)$ is ordinary and
every flex on $C$ is defined over $K^\mathrm{sep}$.
Let $S_C$ be the set of finite places of $K$
where $C$ has a flex locally.
The action of $\Gal(K^\mathrm{sep}/K)$ on the flexes
gives a continuous homomorphism
$\rho_C \colon \Gal(K^\mathrm{sep}/K) \to \AGL_1(\F_3)$.
Then, it is easy to see that $\delta(S_C) = 2/3$
if and only if $\rho_C$ is surjective.
For example, this condition is satisfied for the smooth cubic curve over $\F_3(T)$ defined by
\[ X^3 - T Y^3 - (T^2+1)Z^3 - X^2Z + XY^2 + XZ^2 - TYZ^2 = 0. \]
\end{rem}

\begin{ex}
\label{Example:FlexCharacteristic3}
For a smooth cubic curve $C$ over $\F_3(T)$,
the number of flexes on $C$ defined over $\F_3(T)^{\mathrm{sep}}$
(resp.\ $\F_3(T)^{\mathrm{alg}}$) is denoted by
$N_{\mathrm{sep}}$ (resp.\ $N_{\mathrm{alg}}$).
By Lemma \ref{Lemma:Flex}, $N_{\mathrm{alg}}$ is either $1$ or $3$.
We have $N_{\mathrm{alg}} = 1$ (resp.\ $3$)
if and only if $\Jac(C)$ is supersingular (resp.\ ordinary).
There are $5$ possibilities for the pair $(N_{\mathrm{sep}}, N_{\mathrm{alg}})$:
$(0,1)$, $(0,3)$, $(1,1)$, $(1,3)$, $(3,3)$.
The following table gives examples of smooth cubic curves over $\F_3(T)$
for each of these possibilities:
\[
\begin{array}{|c|l|l|l|}
\hline
& \ \text{Equation} \ & \ \text{Hessian curve} \ & \ \text{Jacobian} \ \\
\hline
\ C_1 \ & XY^2 - X^2Z = Z^3 - TY^3 & X^3 = 0 &
\ \text{supersingular} \ \\
\hline
C_2 & T^2 X^3 + T Y^3 + Z^3 + XYZ = 0 & XYZ = 0 &
\ \text{ordinary} \ \\
\hline
C_3 & Y^2 Z = X^3 + X Z^2 + Z^3 & Z^3 = 0 &
\ \text{supersingular} \ \\
\hline
C_4 & Y^2 Z = X^3 + X^2 Z + T Z^3 & (X - Y)(X + Y)Z = 0 &
\ \text{ordinary} \ \\
\hline
C_5 & Y^2 Z  = X^3 + X^2 Z + Z^3 & (X - Y)(X + Y)Z = 0 &
\ \text{ordinary} \ \\
\hline
\end{array}
\]
\[
\begin{array}{|c|c|c|l|}
\hline
& \ N_{\mathrm{sep}} \ & \ N_{\mathrm{alg}}\  &\ \text{Flexes} \\
\hline
\ C_1 \ & 0 & 1 & \ [0 : 1 : T^{1/3}] \\
\hline
C_2 & 0 & 3 & \ [0 : 1 : -T^{1/3}], \ [1 : 0 : -T^{2/3}],\ [1 : -T^{1/3} : 0] \\
\hline
C_3 & 1 & 1 & \ [0 : 1 : 0] \\
\hline
C_4 & 1 & 3 & \ [0 : 1 : 0],\ [-T^{1/3} : T^{1/3} : 1],\ [-T^{1/3} : -T^{1/3} : 1] \ \\
\hline
C_5 & 3 & 3 & \ [0 : 1 : 0],\ [-1 : -1 : 1],\ [-1:1:1] \\
\hline
\end{array}
\]
\end{ex}

\begin{rem}
In characteristic $3$, the Hessian curve of a smooth cubic curve
always has a singular point
because the determinant of the Hessian matrix splits
into the product of three linear forms;
see Proposition \ref{Proposition:HesseSmoothness}.
\end{rem}

\subsection{Examples of flexes (1)}

Here we give an example of a smooth cubic curve over $\Q$
which shows the bound $8/9$ in Theorem \ref{Theorem:Flex} is optimal.

Let $C \subset \PP^2$ be the smooth cubic curve over $\Q$
defined by
\[ X^3 + 2 Y^3 + 3 Z^3 = 0. \]

Let $S$ be the set of prime numbers $p$ such that
$C$ has a $\Q_p$-rational flex.
The Hessian curve is defined by $XYZ = 0$.
We put
$\alpha := 2^{1/3}$, $\beta := 3^{1/3}$, and $\zeta := \exp(2 \pi \sqrt{-1}/3)$.
The following nine points (in projective coordinates)
are the flexes on $C$:
\begin{center}
\begin{tabular}{ccc}
$[0 : \beta : -\alpha]$, & $[0 : \beta : -\alpha \zeta]$, & $[0 : \beta : -\alpha \zeta^2]$,  \\
$[\beta : 0 : -1]$, & $[\beta : 0 : -\zeta]$, & $[\beta : 0 : -\zeta^2]$,  \\
$[\alpha : -1 : 0]$, & $[\alpha : -\zeta : 0]$, & $[\alpha : -\zeta^2 : 0]$.
\end{tabular}
\end{center}
All of these flexes are defined over $\Q(\alpha,\beta,\zeta)$.
The Galois group
$\Gal(\Q(\alpha,\beta,\zeta)/\Q)$ consists of the elements
$\sigma_{a,b,c}$ with 
$\sigma_{a,b,c}(\alpha) = \alpha \zeta^a$,
$\sigma_{a,b,c}(\beta) = \beta\zeta^b$,
and
$\sigma_{a,b,c}(\zeta) = \zeta^c$
for $a,b \in \{ 0,1,2 \}$ and $c \in \{ 1,2 \}$.
Among them, $\sigma_{1,2,1}$ and $\sigma_{2,1,1}$ do not fix any flex.
Every element other than $\sigma_{1,2,1}$ and $\sigma_{2,1,1}$ fixes 
at least one flex.

Let $S_C$ be the set of places of $\Q$ where $C$ has a flex locally.
By Theorem \ref{Theorem:Chebotarev}, the Dirichlet density $\delta(S_C)$ exists and
$\delta(S_C) = 16/18 = 8/9$.

The action of $\Gal(\overline{\Q}/\Q)$ on the flexes on $C$
corresponds to a continuous homomorphism 
$\Gal(\overline{\Q}/\Q) \to \AGL_2(\F_3)$.
Its image is conjugate to the group $G$ in Example \ref{Example:SmallProportion}
for $(p,r,s) = (3,2,1)$.

\subsection{Examples of flexes (2)}

Here we give another example of a smooth cubic curve over $\Q(\zeta_3)$
which also shows the bound $8/9$ in Theorem \ref{Theorem:Flex} is optimal.
Here $\zeta_3$ is a primitive third root of unity.

Consider the smooth cubic curve $C' \subset \PP^2$ over $\Q(\zeta_3)$ defined by
\[ X^3 + 3 Y^3 + 9 Y Z^2 - 6 Z^3 = 0. \]

Then, it can be checked directly that the image of the absolute Galois group
is a subgroup $G' \subset \AGL_2(\F_3)$ of order $18$
and $\delta(S_{C'}) = 8/9$.
The image of $G'$ by the projection
$\pi \colon \AGL_2(\F_3) \to \GL_2(\F_3)$ is conjugate to the following group:
\[
\bigg\{
\begin{pmatrix}
a & b \\ 0 & a
\end{pmatrix}
\ \bigg| \ a \in \F_3^{\times},\ b \in \F_3
\bigg\}.
\]
In particular, $\pi(G')$ is a subgroup of $\SL_2(\F_3)$.

We can construct similar examples of cubic curves over any global field of characteristic different from 3. Details are left to the reader.

\section{Flexes and Hessians of smooth cubic curves}
\label{Section:FlexHessian}

In this section, we summarize basic results on flexes and Hessian curves of
smooth cubic curves over fields of arbitrary characteristics.
The results in this section are well-known
at least when the characteristic is different from $2,3$.
But the authors could not find reasonably self-contained
references which treat the case of characteristic $2$ or $3$.

\subsection{The case of characteristic different from $2$}

In this subsection, $K$ is a field with $\chara K \neq 2$.
Let $C \subset \PP^2$ be a smooth cubic curves over $K$
whose defining equation is $F(X,Y,Z) = 0$.
The \emph{Hessian matrix} of $F$ is defined by
\[
H_F := \begin{pmatrix}
F_{XX} & F_{XY} & F_{XZ} \\
F_{YX} & F_{YY} & F_{YZ} \\
F_{ZX} & F_{ZY} & F_{ZZ}
\end{pmatrix}.
\]
Here $F_{XX} = \partial_X^2 F,\ F_{XY} = \partial_X\partial_Y F, \dots$ are 
the second order derivatives of $F$.
The \emph{Hessian curve} $\mathrm{He}(C) \subset \PP^2$
of $C$ is the cubic curve over $K$ defined by $\det H_F = 0$.
It is a plane cubic curve over $K$.
It is not smooth in general.

\begin{prop}
\label{Proposition:Flex:CharNot2}
Assume that $\chara K \neq 2$.
\begin{enumerate}
\item The determinant of the Hessian matrix $H_F$ is a non-zero homogeneous polynomial of degree $3$.
Thus, the equation $\det H_F = 0$ defines a cubic curve;
we denote it by $\mathrm{He}(C) \subset \PP^2$.

\item A $K^{\mathrm{alg}}$-rational point $P \in C(K^{\mathrm{alg}})$
is a flex on $C$ if and only if $P$ lies on
the Hessian curve $\mathrm{He}(C)$.
\end{enumerate}
\end{prop}

\begin{proof}
Here, we shall give a brief sketch of the proof.
By a slight abuse of notation, we consider points on $\PP^2$
as column vectors.

We take a $K^{\mathrm{alg}}$-rational point
$P_0 = [X_0: Y_0: Z_0] \in C(K^{\mathrm{alg}})$.
We take a point $P_1 = [X_1: Y_1: Z_1]$ which lies on the tangent line of $C$ at $P_0$, but not on $C$.
We parametrize the points on the tangent line as
$a P_0 + b P_1$ for $[a:b] \in \PP^1(K)$.
Let
    \[
        F(aP_0 + bP_1) =
        a^3F(P_0) + a^2b \, \partial F (P_0) \cdot P_1 +
        ab^2 Q_{F, P_0}(P_1) + b^3 F(P_1)
    \]
be the Taylor expansion of $F(X,Y,Z)$.
Here we put
\begin{align*}
\partial F(P_0) &:= (F_X(P_0), F_Y(P_0), F_Z(P_0)) \\
Q_{F, P_0}(P_1) &:= 2^{-1} \cdot \transp P_1 H_F(P_0) P_1. \\
F_X &:= \partial_X F, \qquad F_Y := \partial_Y F, \qquad F_Z := \partial_Z F,
\end{align*}
where $\transp P_1$ is the transpose of $P_1$ (i.e.\ it is a row vector), and $H_F(P_0)$ is the Hessian matrix at $P_0$.
We have $F(P_0) = 0$ since $P_0$ lies on $C$.
Since $\deg F_X = \deg F_Y = \deg F_Z = 2$,
we have
$\transp P_0 H_F(P_0) = 2 \, \partial F(P_0)$.
If $H_F(P_0)$ has a non-zero kernel and $P_1$ lies in the kernel,
we have
\begin{align*}
  \partial F(P_0) \cdot P_1 &= 2^{-1} \cdot \transp P_0 H_F(P_0) P_1 = 0 \\
  Q_{F, P_0}(P_1) &= 2^{-1} \cdot \transp P_1 H_F(P_0) P_1 = 0.
\end{align*}
(Here we use $\chara K \neq 2$.)
Then we have
$F(aP_0 + bP_1) = b^3 F(P_1)$,
and the tangent line of $C$ at $P_0$ 
meets $C$ at $P_0$ with multiplicity $3$.
Hence the common zero of $H_F$ and $F$ is a flex on $C$.

Conversely, if the line through $P_0$ and $P_1$ is 
a flex line with the flex point $P_0 \in C(k)$
and $P_1 \neq P_0$,
then $F(aP_0 + bP_1)=0$ has a solution at $b=0$
with multiplicity 3.
Thus we have
\[
  F(aP_0 + bP_1) = b^3F(P_1)
  \quad \text{and} \quad
  F(P_0) = \partial F(P_0) \cdot P_1 = Q_{F, P_0}(P_1) = 0.
\]
Therefore, we have
\begin{align*}
    \transp P_0 H_F(P_0) P_0 &= 6 \, F(P_0) = 0, \\
    \transp P_0 H_F(P_0) P_1 &= \transp P_1 H_F(P_0) P_0
= 2 \, \partial F(P_0) \cdot P_1 = 0, \\
    \transp P_1 H_F(P_0) P_1 &= 2 \, Q_{F, P_0}(P_1) = 0.
\end{align*}
Hence we have $\det H_F(P_0) = 0$.
\end{proof}

\subsection{The case of characteristic $2$}

In \cite{Glynn}, Glynn gave a definition of the Hessian curve in characteristic $2$,
which is closely related to the Hessian curve
previously defined by Dickson in 1915 \cite{Dickson}.
Here we shall recall Glynn's results;
see also \cite[Section 9]{AnemaTopTuijp}.
Note that, in characteristic $2$,
we need to modify the definition of the Hessian curve because
the determinant of the Hessian matrix always vanishes.
(In characteristic $2$, the Hessian matrix is an alternating matrix of odd size; its determinant always vanishes.)

Let $K$ be a field of characteristic $2$.
Let $C \subset \PP^2$ be a smooth cubic curve over $K$ defined by the equation
\begin{align*}
a X^3 + b Y^3 + c Z^3 + d X^2 Y + e X^2 Z + f Y^2 X \\
+ g Y^2 Z + h Z^2 X + i Z^2 Y + j X Y Z &= 0
\end{align*}
for $a,b,c,d,e,f,g,h,i,j \in k$.
Then $\Jac(C)$ is ordinary (resp.\ supersingular)
if and only if $j \neq 0$ (resp.\ $j = 0$);
see \cite[Chapter IV, Proposition 4.21]{Hartshorne}.

The \emph{Hessian curve} $\mathrm{He}(C) \subset \PP^2$
of $C$ as a (not necessarily smooth) cubic curve over $K$
defined by
\begin{align*}
 (d e j + a d i + d^2 h + a e g + e^2 f) X^3 \\
+ (f g j + f^2 i + b f h + d g^2 + b e g) Y^3 \\
+ (h i j + e i^2 + c d i + g h^2 + c f h) Z^3 \\
+ (d j^2 + a g j + a f i + a b h + d e g + b e^2) X^2 Y \\
+ (e j^2 + a i j + d e i + a g h + a c f + c d^2) X^2 Z \\
+ (f j^2 + b e j + a b i + b d h + a g^2 + e f g) Y^2 X \\
+ (g j^2 + b h j + b e i + f g h + c f^2 + b c d) Y^2 Z \\
+ (h j^2 + c d j + a i^2 + d h i + a c g + c e f) Z^2 X \\
+ (i j^2 + c f j + f h i + b h^2 + c d g + b c e) Y Z^2 \\
+ (j^3 + a g i + e f i + d g h + b e h + c d f + a b c) X Y Z &= 0.
\end{align*}

\begin{prop}[Glynn]
\label{Proposition:Flex:Char2}
Let $K$ be a field of characteristic $2$.
A point $P \in C(K^{\mathrm{alg}})$ is a flex on $C$
if and only if $P$ lies on the Hessian curve $\mathrm{He}(C)$.
\end{prop}

\begin{proof}
If $\Jac(C)$ is ordinary (i.e.,\ $j \neq 0$),
$\mathrm{He}(C)$ is the curve
\[ \mathcal{C}((A'' + |A|A)/a,\,a^3 + |A|) \]
in the paragraph preceding \cite[Theorem 3.14]{Glynn}.
If $\Jac(C)$ is supersingular (i.e.,\ $j = 0$),
$\mathrm{He}(C)$ is the curve
$\mathcal{C}(ABA,\,|A|)$ in \cite[Theorem 3.14]{Glynn}.
\end{proof}

\subsection{Smooth cubic curves defined by Hesse's normal forms}

Here we give some results on Hesse's normal forms.

\begin{prop}
\label{Proposition:HesseNormalForm}
Let $K$ be an algebraically closed field,
and $C \subset \PP^2$ a smooth cubic curve over $K$.
\begin{enumerate}
\item
Assume that one of the following conditions is satisfied:
\begin{itemize}
\item $\chara K \neq 3$, or
\item $\chara K = 3$ and $\Jac(C)$ is ordinary.
\end{itemize}
Then, under a linear change of variables,
the smooth cubic curve $C$ is mapped into a cubic curve defined by
\[ X^3 + Y^3 + Z^3 + \lambda XYZ = 0 \]
for some $\lambda \in K$.
(It is called \emph{Hesse's normal form}.)

\item 
Assume that $\chara K = 3$ and $\Jac(C)$ is supersingular.
Then the the smooth cubic curve $C$ cannot be mapped into
a cubic curve defined by Hesse's normal form by any linear change of variables.
\end{enumerate}
\end{prop}

\begin{proof}
(1) Essentially the same argument as in the proof of \cite[Lemma 2.1]{ArtebaniDolgachev} works.
Here we briefly give a sketch the proof.

Take two distinct flexes $P_0, P_1 \in C(K)$.
(At this point, we use the assumption that
$\Jac(C)$ is ordinary when $\chara K = 3$.
See Lemma \ref{Lemma:Flex} (3),(4).) 

Let $L \subset \PP^2$ be the line through $P_0$ and $P_1$.
The third intersection $P_2$ of $C$ and $L$ is different from $P_0$ and $P_1$.
Let $G_L \in K[X,Y,Z]$ be the linear form defining $L$.
Let $G_{P_0}, G_{P_1}, G_{P_2} \in K[X,Y,Z]$
be linear forms defining
the tangent lines of $C$ at $P_0,P_1, P_2$, respectively.
Since the divisors on $C$ cut out by $G_L^3$ and $G_{P_0}G_{P_1}G_{P_2}$ are the same, we have
$a G_L^3 + b G_{P_0}G_{P_1}G_{P_2} = F$
for some $a,b \in k^{\times}$.

If $\chara K \neq 3$ and 
the three tangent lines of $C$ at $P_0$, $P_1$, and $P_2$ do not pass through
the same point in $\PP^2$,
as in the proof of \cite[Lemma 2.1]{ArtebaniDolgachev}, we may put
\[
G_L = Z, \quad \text{and} \quad
G_{P_0}G_{P_1}G_{P_2} = X^3 + Y^3 - Z^3 + 3XYZ.
\]
Changing the coordinate $Z$ again, we obtain Hesse's normal form.
If the three tangent lines of $C$ at $P_0$, $P_1$, and $P_2$ pass through
the same point in $\PP^2$,
by changing the coordinates appropriately,
we may assume that
\[ G_L = Z, \quad
   G_{P_0}G_{P_1}G_{P_2} = X^3 + Y^3,
   \quad \text{and} \quad
   a = b = 1. \]
Then, $C$ is the Fermat cubic curve.
In particular, $C$ is written as Hesse's normal form.

Finally, if $\chara K = 3$ and $\Jac(C)$ is ordinary,
we may assume
$G_{P_0} = X$, $G_{P_1} = Y$, and $G_{P_2} = Z$.
Again, by changing the coordinates, we may assume that
$G_L = X+Y+Z$.
Then we obtain Hesse's normal form.

(2)
By Proposition \ref{Proposition:Flex:CharNot2},
a cubic curve defined by Hesse's normal form has
three distinct flexes,
$[1:-1:0]$, $[0:1:-1]$, and $[1:0:-1]$.
By Lemma \ref{Lemma:Flex:Characteristic3} (2),
this never occurs if
$\chara K = 3$ and $\Jac(C)$ is supersingular.
\end{proof}

Let $K$ be an algebraically closed field with
$\chara K \neq 2$.
Let $C \subset \PP^2$ be a smooth cubic curve over $K$,
and $\mathrm{He}(C) \subset \PP^2$ the Hessian curve of $C$.
Let $F_3 \subset \PP^2$ be the Fermat cubic curve defined by $X^3 + Y^3 + Z^3 = 0$.

\begin{prop}
\label{Proposition:HesseSmoothness}
\begin{enumerate}
\item If $\chara K \neq 2, 3$ and $C$ is not isomorphic to $F_3$,
then the Hessian curve $\mathrm{He}(C)$ is smooth over $K$.

\item If $\chara K \neq 2, 3$
and $C$ is isomorphic to $F_3$,
then the Hessian curve $\mathrm{He}(C)$
is the union of three distinct lines in $\PP^2$.

\item If $\chara K = 3$ and $\Jac(C)$ is ordinary,
then the Hessian curve $\mathrm{He}(C)$
is the union of three distinct lines in $\PP^2$.

\item If $\chara K = 3$ and $\Jac(C)$ is supersingular,
then the Hessian curve $\mathrm{He}(C)$ is
a non-reduced cubic curve defined by the third power of a linear form.
\end{enumerate}
\end{prop}

\begin{proof}
First we treat the case when $\chara K \neq 2,3$. 
Since Hessian of a cubic curve $C$ is a covariant,
by Proposition \ref{Proposition:HesseNormalForm},
we may assume $C$ is defined by a Hesse normal form
$X^3 + Y^3 + Z^3 + \lambda XYZ = 0$.
It is smooth if and only if $\lambda^3 \neq -27$
(regardless of $\chara K$).
Its Hessian form is
\[ -6\lambda^2(X^3 + Y^3 + Z^3) + (216 + 2\lambda^3)XYZ. \]
Thus, $\mathrm{He}(C)$ is singular
if and only if
\[
  -6\lambda^2 = 0
  \quad \text{or} \quad
  (216 + 2\lambda^3)^3 = -27 \cdot (-6\lambda^2)^3.
\]
This occurs precisely when
$\lambda \in \{ 0,\, 3\zeta_3^i,\, 6 \zeta_3^i \}$,
where $\zeta_3$ is a primitive third root of unity, and $0 \le i \le 2$.
The first case $\lambda = 0$ is the Fermat cubic curve $F_3$,
the second case $\lambda = 3\zeta_3^i$ is singular,
and the third case $\lambda = 6\zeta_3^i$ is isomorphic to $F_3$ via
\[
  x \mapsto (x + y + z)/3, \quad
  y \mapsto (x + \zeta_3y + \zeta_3^2 z)/3, \quad
  z \mapsto (x + \zeta_3^2y + \zeta_3z)/3.
\]
Since the singular curves defined by a Hesse normal form is
union of three distinct lines, the assertions (1) and (2) are proved.

We shall consider the case $\chara K = 3$.
If $\Jac(C)$ is ordinary,
we may assume that $C$ is defined by a Hesse normal form
by Proposition \ref{Proposition:HesseNormalForm} (1).
The Hessian is $2\lambda^3 XYZ$.
Since $C$ is smooth, we have $\lambda \neq 0$, and the assertion follows.
If $\mathrm{Jac}(C)$ is supersingular,
$C$ is defined by
$Y^2Z - X^3 - XZ^2 - Z^3 = 0$
by a linear change of coordinates.
The Hessian of this curve is $Z^3 = 0$.
\end{proof}

\begin{rem}
The above proof does not work in characteristic $2$
because the Hessian in characteristic $2$ defined by Glynn
is not a covariant.
(See \cite{Glynn} for details.)
\end{rem}

\subsection*{Acknowledgements}

The authors would like to thank Jean-Pierre Serre
for providing a proof of group theoretic results needed to show the Hasse principle for flexes.
The proof of Proposition \ref{Proposition:AffineTransformation}
in the current version is due to him.
He also suggested to study the Hasse principle allowing exceptional sets of positive density.
(See Remark \ref{Remark:Serre}.)

The authors would also like to thank anonymous referees
for giving helpful comments,
giving the reference \cite{Milne}.
The authors corrected and improved the proof of
Proposition \ref{Proposition:AffineTransformation2}
according to the comments of the referees.

The work of Y.\ I.\ was supported by
JSPS KAKENHI Grant Number 16K17572 and 21K18577.
The work of T.\ I.\ was supported by JSPS KAKENHI Grant Number 20674001,
26800013, 21H00973, and 21K18577.
This work was supported by the Sumitomo Foundation FY2018 Grant
for Basic Science Research Projects (Grant Number 180044).

In the course of obtaining the main results of this paper,
experimental calculations of subgroups of affine transformation groups
by \texttt{GAP} (version 4.8.10) \cite{GAP} were very helpful.


\begin{thebibliography}{99}

\bibitem{AnemaTopTuijp}
  Anema, A.\ S.\ I., Top, J., Tuijp, A.,
  \textit{Hesse pencils and $3$-torsion structures},
  SIGMA 14 (2018), Paper No.\ 102, 13 pp.

\bibitem{ArtebaniDolgachev}
  Artebani, M., Dolgachev, I.\ V.,
  \textit{The Hesse pencil of plane cubic curves},
  Enseign.\ Math.\ (2) \textbf{55} (2009), no.\ 3--4, 235--273.

\bibitem{Ballot}
  Ballot, C., \textit{Competing prime asymptotic densities in $\mathbb{F}_q[X]$: a discussion},
  Enseign.\ Math.\ (2) \textbf{54} (2008), no.\ 3--4, 303--327.

% \bibitem{Bhargava:HassePrinciple}
%   Bhargava, M, \textit{A positive proportion of plane cubics fail the Hasse principle}, preprint, 2014, \texttt{arXiv:1402.1131}.

\bibitem{CameronCohen}
  Cameron, P.\ J., Cohen, A.\ M.,
  \textit{On the number of fixed point free elements in a permutation group},
  A collection of contributions in honour of Jack van Lint,
  Discrete Math.\ 106/107 (1992), 135--138.

\bibitem{Dickson}
  Dickson, L.\ E.,
  \textit{Invariantive theory of plane cubic curves modulo 2},
  Amer.\ J.\ Math.\ 37 (1915), 107--116.
 
\bibitem{Dolgachev:ClassicalAlgebraicGeometry}
  Dolgachev, I.\ V., \textit{Classical algebraic geometry.\ A modern view},
  Cambridge University Press, Cambridge, 2012.

% \bibitem{DvornicichPaladino}
%   Dvornicich, R., Paladino, L.
%   \textit{Local-global questions for divisibility in
%   commutative algebraic groups},
%   preprint, 2018, \texttt{arXiv:1706.03726}.

\bibitem{DvornicichZannier}
  Dvornicich, R., Zannier, U.,
  \textit{Local-global divisibility of rational points in some commutative algebraic groups},
  Bull.\ Soc.\ Math.\ France 129 (2001), no.\ 3, 317--338.

\bibitem{FriedJarden}
  Fried, M.\ D., Jarden, M., \textit{Field arithmetic}, Third edition,
  Ergebnisse der Mathematik und ihrer Grenzgebiete.\ 3.\ Folge.\ A Series of Modern Surveys in Mathematics, 11.\ Springer-Verlag, Berlin, 2008.

\bibitem{Glynn}
  Glynn, D.\ G., \textit{On cubic curves in projective planes of characteristic two},
  Aust.\ J.\ of Comb., 17 (1998), 1--20.

\bibitem{Greenberg:RationalPoints}
  Greenberg, M.\ J., \textit{Rational points in Henselian discrete valuation rings},
  Inst.\ Hautes \'Etudes Sci.\ Publ.\ Math.\ No.\ 31 (1966), 59--64.

\bibitem{Hartshorne}
  Hartshorne, R., \textit{Algebraic geometry}, Graduate Texts in Mathematics, No.\ 52.,
  Springer-Verlag, New York-Heidelberg, 1977.

\bibitem{IshitsukaIto:LocalGlobal}
  Ishitsuka, Y., Ito, T.,
  \textit{The local-global principle for symmetric determinantal representations of smooth plane curves}, Ramanujan J.\ 43 (2017), no.\ 1, 141--162.

\bibitem{IshitsukaIto:LocalGlobalChar2}
  Ishitsuka, Y., Ito, T.,
  \textit{The local-global principle for symmetric determinantal representations of smooth plane curves in characteristic two},
  J.\ Pure Appl.\ Algebra 221 (2017), no.\ 6, 1316--1321.

%\bibitem{IIOTU}
%   Ishitsuka, Y., Ito, T., Ohshita, T., Taniguchi, T., Uchida, Y.,
%   \textit{The local-global property for bitangents of plane quartics},
%   JSIAM Lett.\ 12 (2020), 41--44.

\bibitem{JahnelLoughran}
  Jahnel, J., Loughran, D.,
  \textit{The Hasse principle for lines on del Pezzo surfaces},
  Int.\ Math.\ Res.\ Not.\ 2015, no.\ 23, 12877--12919.

\bibitem{Milne}
  Milne, J.\ S., \textit{Arithmetic duality theorems},
  Second edition.\ BookSurge, LLC, Charleston, SC, 2006.

\bibitem{NeukirchSchmidtWingberg}
  Neukirch, J., Schmidt, A., Wingberg, K.,
  \textit{Cohomology of number fields}, Second edition,
  Grundlehren der Mathematischen Wissenschaften, 323.\ Springer-Verlag, Berlin, 2008.

% \bibitem{Selmer}
%   Selmer, E.\ S.,
%   \textit{The Diophantine equation $ax^3 + by^3 + cz^3 = 0$},
%   Acta Math.\ 85 (1951), 203--362.

\bibitem{Serre:Jordan}
  Serre, J.-P., \textit{On a theorem of Jordan},
  Bull.\ Amer.\ Math.\ Soc.\ (N.S.) 40 (2003), no.\ 4, 429--440.

\bibitem{Sonn}
  Sonn, J., \textit{Polynomials with roots in $\mathbb{Q}_p$ for all $p$},
  Proc.\ Amer.\ Math.\ Soc.\ 136 (2008), no.\ 6, 1955--1960.

\bibitem{Wong}
  Wong, S.,
  \textit{Power residues on abelian varieties},
  Manuscripta Math.\ 102 (2000), no.\ 1, 129--138.

\bibitem{GAP}
  The GAP Group, GAP - Groups, Algorithms, and Programming,
  Version 4.8.10, 2018, \texttt{https://www.gap-system.org}

\end{thebibliography}
\end{document}